\documentclass[12pt]{amsart}
\usepackage{amssymb,amscd}
\usepackage{eucal,mathrsfs}
\usepackage{bbm}
\usepackage[all]{xy}
\usepackage{mdwlist}

\newcommand{\Z}{\mathbb {Z}}

\newcommand{\A}{\mathbb {A}}

\newcommand{\kk}{\mathbbm{k}}
\newcommand{\aaa}{\mathfrak{a}}
\newcommand{\be}{\mathfrak{b}}

\newcommand{\gee}{\mathfrak{g}}

\newcommand{\pe}{\mathfrak{p}}

\newcommand{\M}{\mathfrak{M}}
\newcommand{\m}{\mathfrak{m}}
\newcommand{\n}{\mathfrak{n}}
\newcommand{\Oo}{\mathcal{O}}

\newcommand{\Ee}{\mathcal{E}}
\newcommand{\Eff}{\mathcal{F}}
\newcommand{\Ge}{\mathcal{G}}

\newcommand{\rad}{\operatorname{rad}}
\newcommand{\ord}{\operatorname{ord}}

\newcommand{\Ext}{\operatorname{Ext}}

\newcommand{\Ker}{\operatorname{Ker}}

\newcommand{\End}{\operatorname{End}}

\newcommand{\Spec}{\operatorname{Spec}}

\newcommand{\Spm}{\operatorname{Spm}}

\newcommand{\isoto}{\cong}

\newcommand{\Iscr}{\mathscr{I}}

\newcommand{\gr}{\operatorname{gr}}

\newcommand{\defeq}{\overset{\operatorname{def}}{=}}

\newcommand{\Jot}{\mathfrak J}

\DeclareMathOperator{\pd}{pd}
\DeclareMathOperator{\depth}{depth}

\newcommand{\Serre}[2]{{#1}_{\bar X}{\langle #2 \rangle}}

\theoremstyle{plain}
\newtheorem{theorem}{\indent\bf Theorem}[section]
\newtheorem{lemma}[theorem]{\indent\bf Lemma}
\newtheorem{corollary}[theorem]{\indent\bf Corollary}
\newtheorem{proposition}[theorem]{\indent\bf Proposition}

\theoremstyle{definition}
\newtheorem{definition}[theorem]{\indent\sc Definition}

\newtheorem{example}[theorem]{\indent\sc Example}

\begin{document}

\title
[Auslander regularity of norm based extensions]
{Auslander regularity of norm based extensions of Weyl algebras}
\author{Yoshifumi Tsuchimoto}
\footnote{2000\textit{ Mathematics Subject Classification}.
 Primary 14R15; Secondary 14A22.}
\begin{abstract}
We first prove that
 for a Weyl algebra over a field of positive characteristic,
its norm based extension is locally Auslander regular. 
We then prove that given an algebra which is Zariski locally isomorphic to the 
Weyl algebra, its norm based extension similarly defined is locally 
Auslander regular if and only if it is isomorphic to the original Weyl algebra.
\end{abstract}

\maketitle
\section{introduction}
Let $A$ be a Weyl algebra over a field of positive characteristic. 
It is finite over its center $Z(A)$ and may thus studied by looking at 
the sheaf $A_X$ of algebras over $X=\Spec(Z(A))\isoto \A^{2n}$.
It is known that 
the object is deeply related to Dixmier conjecture and Jacobian problem
(See for example \cite{T4}\cite{T5}\cite{Kontsevich2}\cite{AdjamagboEssen}.)
In \cite{T9}, the author defined the ``norm based extension'' (NBE) 
$A_{\bar X}$ of $A_X$ as an algebra sheaf over a projective space
$\bar X=\mathbb P^{2n}$. 
In this paper we show that the norm based extension
$A_{\bar X}$  is locally Auslander regular,
in the sense that every stalk of it is Auslander regular.

Next we take an algebra $B$ which is Zariski locally 
isomorphic to $A$, that means, $B$ is an algebra over $Z$ 
whose corresponding sheaf $B_X$ on $X$ is Zariski locally isomorphic to $A_X$. 
Such algebra is considered for example in \cite{Kontsevich1}.
Then we show that the norm based extension $B_{\bar X}$  defined in 
a similar manner as $A_{\bar X}$ 
is locally Auslander regular if and only if $B$ is isomorphic to $A$
as a $Z$-algebra.

\section{local rings in the sense of this paper}
In this paper, by a ring we mean a unital associative ring.
By a module we mean a left module unless otherwise specified.

Unfortunately, there are several definitions of "local rings". 
In several literature, 
the term ``local ring'' is meant to stand for a ring with only one
maximal left-ideal. See for example \cite[Proposition 10.1.1]{HGK} for
details. 
In contrast, we are interested in the rings with only one
maximal (two-sided) ideal.
So we introduce here somewhat unusual terms  
`demi local rings', `NAK local rings', and,
following Ramras,  a term `quasi local rings' to make a distinction. 

In this paper, we mainly deal with rings $A$ which are finite over 
their central subring $R$.
In such cases, the two theory of `local rings' are not very far.
Indeed, as we will illustrate in Lemma \ref{lemma:central},
 maximal left ideals of 
$A$ are  pretty well understood by looking at maximal ideals.

\subsection{demi local rings and quasi local rings.}
\label{subsection:localring}
\begin{definition}
Let us call a unital associative ring $A$ 
a {\it demi local ring} if $A$ has only one maximal ideal $\m$.
As is customary, 
we often say that $(A,\m)$ is a demi local ring to express this fact.
\end{definition}

\begin{definition}
Let $I$ be an ideal of a ring $A$. 
We say that {\it Nakayama's lemma holds} for $(A,I)$ if 
for any finitely generated left $A$-module $M$, we have
$$
I M=M \ \implies \ M=0.
$$ 
\end{definition}
For any ring $A$,  let us denote by $\rad(A)$ its Jacobson radical.
Namely, $\rad(A)$ is the intersection of all maximal left ideal of $A$.
It is well known that for any ring $A$, its radical is in fact an ideal of $A$
and that the Nakayama's lemma holds for $(A,\rad(A))$. 
(See for example \cite[Lemma 3.4.4, Lemma 3.4.11]{HGK})

We note:
\begin{lemma} \label{lemma:NAK} 
Let $(A,\m)$ be a demi local ring. The following conditions are equivalent.
\begin{enumerate}
\item Nakayama's lemma holds for $(A,\m)$.
\item Every maximal left ideal $J$ of $A$ contains $\m$ as a submodule.
\item $\m =\rad(A)$.
\end{enumerate}

\end{lemma}
\begin{proof}
That $ (3)$ implies $(1)$ is a consequence of the usual 
Nakayama's lemma cited above. 

$(1)\implies(2)$: Let $J$ be a maximal left ideal of $A$.
If $J\not \supset \m$, then by the maximality of $J$ we have
$ J+\m=A$. In other words, we have $\m A/J=0$. By the assumption we have 
$A/J=0$, which is a contradiction.

$(2) \implies (3)$ is a consequence of the definition of $\rad(A)$
and the fact that $\rad(A)$ is an ideal of $A$.
\end{proof}

\begin{definition}
Let us call a demi local ring $(A,\m)$ a {\it NAK local ring} if
it satisfies the equivalent conditions of Lemma \ref{lemma:NAK}.
If in addition $A/\m$ is simple Artinian, we call it a {\it quasi local ring}.
\end{definition}
The term quasi local ring appears in papers of Ramras. 
See for example \cite{Ramras1}.
The term also appears in papers such as \cite{DAJordan1},\cite{BHM}.

\subsection{demi local rings which are finite over the center}

We are mainly interested in cases where our ring $A$ is 
finite over its center $Z(A)$.
For technical reasons we consider a slightly
general situation and employ a central subring $R$ of $A$ (that means,
a subalgebra of $Z(A)$)  instead of $Z(A)$ itself.
\begin{lemma} 
\label{lemma:central}
Let $(A,\m)$ be a demi local ring which is finite over
its central subring $R$.
Then:
\begin{enumerate}
\item $\m_R=\m\cap R$ is the unique maximal ideal of $R$ 
(so that $(R,\m_R)$ is a local ring 
in the usual sense in commutative algebra).  
\item $A/\m$ is isomorphic to a full matrix algebra $M_{k_0}(D)$ over a
skew field $D$. (In other words, $A/\m$ is simple Artinian.)
\item If furthermore $A$ is a quasi local ring, then
maximal left ideals of $A$  in one to one correspondence with 
points of the projective space $\mathbb P^{k_0-1}(D)$ of 
dimension $k_0-1$ over $D$.
\end{enumerate}

\end{lemma}
\begin{proof}

(1)
For any maximal ideal $\mathfrak a$ of $R$, we have
$$
\mathfrak a A  \neq A
$$
by the usual Nakayama's lemma for the commutative ring $(R,\m_R)$. 
By the maximality of $\m$ we have $\mathfrak a A \subset \m$, 
so that we have  $\mathfrak a \subset \m \cap R=\m_R$. 
By the maximality of $\mathfrak a$ we have $\mathfrak a=\m_R$.

(2) $A/\m$ is a finite dimensional simple algebra over a field $R/\m_R$.
Thus we may use Wedderburn's theorem. 

(3) Since every maximal left ideal $J$ contains 
 $\m$, it corresponds bijectively to a
maximal left ideals of $A/\m\isoto M_{k_0}(D)$.
\end{proof}

\begin{lemma}
 \label{property:NAK}
Let $(A,\m)$ be a demi local ring which is 
finite over its central subalgebra $R$.
We put $\m_R=\m\cap R$. 
(Note that $(R,\m_R)$ is a local ring in view of Lemma \ref{lemma:central}.)
Then $(A,\m)$ is quasi local if and only if 
 the $\m$-adic topology on $A$
coincides with the $\m_R A$-adic topology. Namely, if 
$\m^{k_0} \subset \m_R A $ for
some $k_0 \gg 0$.
\end{lemma}
\begin{proof}
By Lemma \ref{lemma:central} we see that $A/\m$ is simple Artinian.

Assume $\m^{k_0} \subset \m_R A $ for
some $k_0 \gg 0$.
Since $A$ is finite over $R$, 
$M$ is a finite $R$-module. On the other hand $M=\m M$ implies $M=\m^{k_0} M$, 
so that we have
$M=\m_R M$. Thus we see that $M=0$ by using 
the usual Nakayama's lemma for $(R,\m_R)$.

Conversely, let us assume $(A,\m)$ is quasi local. 
$A/\m_RA$ is a module of finite rank over a field $R/\m_R$.
Thus a decreasing sequence $\{\m^k (A/\m_RA)\}_{k=1}^\infty$ stops. 
In other words, there exists some integer $k_0$ such that
$$
N=\cap_{k} \m^k (A/\m_R A) =\m^{k_0} (A/\m_RA)
$$
holds. The module $N$ clearly is a finitely generated over $A$ and 
satisfies $\m N=N$. By the assumption we have $N=0$.
\end{proof}

\noindent
Note:
There exists a demi local ring finite over its center which 
is not a quasi local ring.
Indeed, let us consider a semisimple Lie algebra $\gee$ over a field
$\\k$ of positive characteristic 
and consider its universal enveloping algebra $U=U(\gee)$.
Then $U$ is finite over its center. The localization $A=U_{(0)}$ of $U$
at the origin is a demi local ring but is not quasi local, since
we have 
$
[\gee,\gee]=\gee
$
so that the unique maximal ideal $\m$ of $A$ satisfies $\m^2=\m$.

\subsection{Rings which are finite over their centers}
In the world of  commutative algebraic geometry, 
it is useful to consider commutative rings at each of their maximal ideals.
We localize a commutative ring $R$ at its maximal ideal $\m$ and
obtain the localization $R_{\m}$. $R_{\m}$ carries the local information
of $R$ around $\m$.
It is also useful to consider the completion $\hat R_{\m}$ to
obtain the local information. 
We would like to extend this idea to deal with non commutative rings 
which have sufficiently many maximal ideals.

Let $A$ be a ring with a maximal ideal $\M$. 
We use its $\M$-adic completion
$$
\hat A_{\M}=\varprojlim_n A/\M^n
$$
as a substitute for ``localization of $A$ at $\M$.''
However,
when we deal with non commutative rings, 
whether we use localization or completion,
these notions are sometimes insufficient.
For example, when $A$ is equal to $ U(\gee)$, the 
universal enveloping algebra of a simple Lie algebra $\gee$ over a field
$\kk$ of characteristic $p\neq 0$, then a maximal ideal $\M=\gee U(\gee)$
satisfies $\M^2=\M$, which is in deep contrast with the NAK lemma for 
quasi local ring as developed in subsection \ref{subsection:localring}.  As a result, the completion
$\varprojlim_{n} A/\M^n$ is equal to the ``small'' residue ring $A/\M$
 in this case.   

Localizations and formal completions of such algebras behaves much like 
the ones in the commutative case.
Let us  distinguish between such behaviors.
The idea is follows. Given a ring $A$ 
as in the definition, we may surely localize it with respect to $\m$ 
and obtain $A\otimes_R R_\m$. We may also 
consider the completion $A\otimes_R \hat R_{\m}$.
Such localizations and completions are  common in the usual
 commutative algebraic geometry. 
Then we examine if the completion $A \otimes_ R R_\m$ is a good one. 
We first note:
\begin{lemma} 
Let $A$ be a ring which is finite over its central subring $R$.
Then:
\begin{enumerate}
\item  For any maximal ideal $\M$ of $A$, $\aaa=\M\cap R$ is a maximal ideal
of $R$.
\item For any maximal ideal $\m$ of $R$, the set
$$
S=\{\M\in \Spm(A) ; \M \cap R=\m\}
$$
of maximal ideals which lie over $\m$ is a finite set.
\item 
 The intersection $\mathfrak N=\cap_{\M \in S}\M$
 satisfies
$
\mathfrak N^{k_0}\subset \m A
$
for some positive integer $k_0$.
\end{enumerate}
\end{lemma}
\begin{proof}
\noindent (1)
By using the usual Nakayama lemma for a finite module $R/\M$ over 
the commutative ring $R$, we see that there exist  a maximal left ideal $\be$
of $R$ such that $(A/\M)\be \neq A/\M$. It follows that 
$\be \subset \M \cap R=\aaa$. By using the maximality of $\be$ we have 
$\aaa=\be$.

\noindent (2)
The algebra $\bar A=A/\m$ is an Artinian ring which is finite over a field 
$\bar R=R/\m$.
Thus we see that the set $S$ is a finite set. 

\noindent (3) According to a theorem of Hopkins \cite[Theorem 3.5.1]{HGK},
the Jacobson radical $\rad(\bar A)$ of $\bar A$ is nilpotent. It follows  that
any maximal ideal $\bar \M$ of $\bar A$ contains $\rad(\bar A)$ as a subset. 
On the other hand, 
according to a theorem \cite[Theorem 3.5.5]{HGK} 
on semisimple Artinian algebras, 
we see that $\bar A/\rad(\bar A)$ is a direct sum
of full matrix algebras over skew fields. 
We therefore easily see that $\bar {\mathfrak N}=\rad(\bar A)$ 
and  $\bar {\mathfrak N}$ is therefore nilpotent.

\end{proof}

We then employ the following definition.
\begin{definition}
Let $A$ be a ring which is finite over its central subring $R$.
Let $\m$ be a maximal ideal of $R$.
We know that there exists only a finite maximal ideals 
$\M_1,\M_2,\dots, \M_s$.
Let us say that 
$A$ has a {\it formally completely decomposable fiber} at $\m$ if 
the natural homomorphism
$$
A\otimes _R \hat R_{\m}
\to 
\prod_{j} \hat A_{\M_j}
$$
is bijective.
When there exists only one maximal ideal of $A$ which lies over $\m$,
(that means, when $s=1$), then we say that $A$ has 
a {\it formally indecomposable fiber} at $\m$.

\end{definition}
\begin{note}
In the above definition, if $A$ has formally completely decomposable 
fiber at $\m$,
then each $\hat A_{M_j}$ is flat over $A$. Indeed, $A\otimes_R \hat R_\m$
is flat over $A$ as the usual commutative ring theory tells.
$\hat A_{M_j}$ is a direct summand of $A\otimes_R \hat R_\m$ as an $A$-module
and is hence an $A$-flat module.

\end{note}

The author admits that the definition above is a temporary one.
There should be a better way to manage the situation
which is independent of the choice of the
center $R$.

As a result of  a usual argument of topology, we obtain the following:
\begin{lemma}
Let $A$ be a ring which is finite over its central subring $R$.
Let $\m$ be a maximal ideal of $R$.
Let $\M_1,\M_2,\dots, \M_s$ be the maximal ideals which lies over $\m$.
Then $A$ has a formally completely decomposable fiber at $\m$ if and 
only if the following condition is satisfied.
$$ 
\forall k>0 \exists N>0 
; \quad (\M_1^N \cap \dots \cap \M_s ^N ) \subset (\M_1\cap \dots \cap \M_s)^k
$$

\end{lemma}
\qed

\begin{corollary}
Let $A$ be a ring which is finite over its central subring $R$.
Let $\m$ be a maximal ideal of $R$.
If there exists only one  maximal ideal of $A$ which lies over $\m$,
then 
$A$ has a formally indecomposable fiber at $\m$.
\end{corollary}
\qed

\begin{proposition}
If $A$ is a commutative  ring which is finite over a noetherian 
subring $R$, then $A$ has a formally completely decomposable fiber at every closed point of $R$.
\end{proposition}
\begin{proof}
This is a direct consequence of the NAK lemma for commutative rings.
\end{proof}

\begin{example}
Let $R=\kk$ be field.
Let $A\subset M_n(\kk)$ be the ring of upper triangular matrices over $R$:
$$
A=\{(a_{ij})\in M_n(R)|\   a_{ij}=0 \text{ if } i> j\}
$$
Then the maximal ideals which are on $\m$ are
$\M_1,\M_2,\dots, \M_n$ defined as follows.
$$
\M_k=\{(a_{ij})\in A; a_{kk}=0\}
$$
We then note that 
$A\otimes_\kk \hat \kk_{\m}= A$ does not decompose into a direct sum of rings.
$A$ in this case therefore does not have completely decomposable fiber at $\m$.
It might be also useful to note that $\M_k^2=\M_k$ holds for any $k$ so that
each  completion $\varprojlim A/\M_k^{n}$ is isomorphic to  a rather ``small'' 
module
$A/\M_k(\isoto \kk)$.

\end{example}

\section{Auslander regularity}
In this section we summarize some results concerning Auslander regularity of
non commutative rings. It may be helpful to read
 an article\cite{BE1} for an introduction to the topic.

Let us begin by giving the definition of Auslander regularity.
\begin{definition}
Let $A$ be a ring Then the grade $j(M)$ of a left $A$-module $M$ is defined by
$$
j(M)=\inf\{ i| \Ext^i_A (M,A)\neq 0\}.
$$
(By convention we define $\inf\emptyset=+\infty$.)
\end{definition}

\begin{definition}
Let $A$ be a ring. Then:
\begin{enumerate}
\item  We say that $A$ {\it satisfies the Auslander condition}
 if for every Noetherian left
$A$-module $M$ and for all $i \geq 0$, $j(N)\geq i$ for all right -$A$ submodules
$N \subset \Ext^i_A(M,A)$.
\item  $A$ is said to be  {\it Auslander-Gorenstein}
 if $A$ is left Noetherian,
satisfies the Auslander condition, and has finite left injective
dimension;
\item  $A$ is said to be {\it Auslander regular}
 if it is Auslander-Gorenstein and has finite global
dimension.
\end{enumerate}
\end{definition}
 A commutative ring is Auslander regular 
if and only if it
is a regular ring in the usual sense of commutative ring theory.
It is also easy to see that the following lemma holds.
The author could not find a right reference.
For the sake of self-containedness, let us give here a proof.
\begin{lemma}
Auslander regularity is Morita invariant.
\end{lemma}
\begin{proof}
Assume that 
$A$ is Morita equivalent to $B$  and 
that 
$B$ is Auslander regular.
There exists a (left $A$-, right -$B$) module $P$  
and a (left $B$-, right -$A$) module $Q$  
such that they give an 
equivalence of categories
\begin{equation*}
\xymatrix
{
(\text{left }A\text{ modules})
\ar@/^/[r]^{Q\otimes_ A \bullet} & \ar@/^/[l]^{P\otimes_B \bullet}
(\text{left }B\text{ modules})
}
\tag {\textcircled{L}}
\end{equation*}
and another equivalence of categories
\begin{equation*}
\xymatrix
{
(\text{right }A\text{ modules})
\ar@/^/[r]^{\bullet \otimes_ A P} & \ar@/^/[l]^{\bullet \otimes_B Q}
(\text{right }B\text{ modules})
}
.
\tag {\textcircled{R}}
\end{equation*}

Let $M$ be a finitely generated left $A$-module. 
Let $N$ be an $A$-right submodule of $ \Ext_A^i(M,A)$ for some integer $i$.
The equivalence \textcircled{L} of categories gives rise to
an isomorphism
$$
\Ext_A^i(M,A)\isoto 
\Ext_B^i(Q\otimes_A M,Q)
$$
of  extension groups
for each $i$.
These extension groups are  right -$A$  modules and the isomorphism is
a isomorphism of right -$A$-modules.
The right hand side corresponds, via the equivalence \textcircled{R} to a
right -$B$ module
$$
\Ext^i_A(Q\otimes_A M,Q)\otimes_A P
\isoto
\Ext_B^i(Q\otimes_A M , Q\otimes_A P).
$$
Note that $Q\otimes_A P \isoto B$.
So the inclusion $N\subset \Ext^i_A(M,A)$ corresponds via \textcircled{R}
to an inclusion
$$
N\otimes_A P\subset \Ext_B^i(Q\otimes_A M,B).
$$
The Auslander condition for $B$ implies
$$
\Ext_B^j(N\otimes_A P, B) =0 \quad (\forall j \leq i).
$$
Now by using the same  idea as above and we obtain
\begin{align*}
&Q\otimes_A \Ext^j_A(N,A) \\
\isoto
&Q\otimes_A\Ext_B^j(N\otimes_A P,P) \\
\isoto 
&\Ext^j_B(N\otimes_ A P,Q\otimes _A P)\\
&=0 \quad \text{(if $i \geq j$.)} 
\end{align*}
An $A$-module $\Ext^j_A(N,A)$ corresponds
to $Q\otimes_A \Ext^j_A(N,A)$ 
and is therefore equal to zero if $j\leq i$.
Thus $A$ satisfies the Auslander condition.

That $A$ has a finite projective dimension is obvious.
 
\end{proof}

\begin{corollary}\label{corollary:auslanderregular}
 A full matrix algebra $M_s(\kk[t_1,\dots t_n])$ over a
polynomial algebra over a field $\kk$ is Auslander regular.
\end{corollary}

The Auslander regularity might be 
a little difficult to grasp at a first glance.
Fortunately, there is a good criterion for the Auslander regularity
due to Bj\"ork. 
Namely, if $A$ has a ``good'' filtration such that 
the associated algebra $\gr A$ is Auslander regular, then $A$ itself
is Auslander regular. We may thus reduce the problem of judging 
the Auslander regularity of $A$ to the same problem for $\gr A$.
Let us now begin by defining what ``good'' filtrations are.
\begin{definition}
By a {\it filtration} of a ring $A$ we mean an increasing sequence
$\{A_i\}$ of additive subgroup of $A$ such that
$$
A=\bigcup_{i=0}^\infty A_i, \quad
 \bigcap_{i=0}^\infty A_i=0, \quad 
A_n A_m \subset A_{n+m} \quad(\forall n,m), \quad
 1\in A_0 
$$
holds.
We say that a {\it strong closure condition} is satisfied if
for any finite elements $a_1,a_2,\dots,a_s$ of $A$ and for any
integers $i_1,i_2,\dots,i_s$, submodules
$$
A_{i_1} a_{i_1} + \dots A_{i_s} a_{i_s}
$$
$$
a_{i_1}A_{i_1}  + \dots + a_{i_s} A_{i_s}
$$
are closed in the filtrated topology.
\end{definition}
It might be helpful to note that the associated graded ring $\gr A$ is 
denoted by $G A$ in several papers like \cite{Bjork1} ,\cite {BE1}.
There are several other ``closure conditions''. These conditions are
verified easily for the rings of our concern:
\begin{lemma} \label{lemma:closure}
Let $(A,\m)$ be a 
quasi local ring finite over its central Noetherian subalgebra $R$,
If $A$ is filtrated such that its topology coincides with 
a $\n$-adic topology for some ideal $\n \subset \m$, 
then $A$ with the $\m$-adic filtration satisfies the strong closure condition.
\end{lemma}
\begin{proof}
By imitating the proof of Lemma \ref{property:NAK}, we see that 
there exists a positive integer $k_0$ such that 
$\n^{k_0}\subset \n \cap R$ and that $\n$-adic topology coincides with
$\n\cap R$-adic topology. The lemma then follows from  the Krull intersection 
theorem for commutative rings.
\end{proof}

We now cite the following strong result of Bj\"ork.
\begin{theorem}[Bj\"ork]\cite{Bjork1}\label{theorem:Bjork}
Let $A$ be a filtered ring such that ``closure conditions" hold.
If the associated graded algebra $\gr(A)$ is Auslander regular, 
then the original ring $A$ is also Auslander regular.
\end{theorem}

By using Lemma\ref{lemma:closure}, 
we obtain immediately the following corollary.

\begin{corollary}
Let $(A,\m)$ be a Noetherian quasi local ring finite over its center.
Then $A$ is Auslander regular if 
the associated graded ring $\gr_\m A$ is Auslander
regular.

\end{corollary}

The Auslander regularity of a ring guarantees us several good properties 
of the ring.
The following theorem of Ramras is an example. 
It  plays an important role in our argument.

\begin{theorem}\cite[Theorem2.16]{Ramras1}\label{theorem:Ramras}
Let $A$ be a quasi local ring finite over a central regular local ring
$R$ of dimension $n$. 
If the left injective dimension of $A$ over $A$ is finite, then $A$ is 
$R$-free.

\end{theorem}


%
%
The following fact seems to be well-known to the specialists.
\begin{proposition}
Let $A$ be a non commutative Noetherian ring. 
Then $A$ is an Auslander regular ring 
if and only if the following conditions are fulfilled.
\begin{enumerate}
\item The global dimension of $A$ is finite.
\item Let
$$
0 \to A \to E_0 \to E_1 \to \dots,
$$
be the minimal injective resolution 
of $A$ as an $A$-left module. 
Then the flat dimension of each $E_i$  is less than or equal to $i$.
\end{enumerate}
\end{proposition}
For the proof see for example \cite{Miyachi1}.
\qed

On the other hand, 
we know
 that taking injective hulls commutes with localizing modules
with respect to a multiplicative subset of $R$, 
namely, the following lemma holds.
\begin{lemma}
Let $A$ be a ring which is finite over its noetherian central subalgebra $R$.
Let $S$ be a multiplicative subset of $R$.
Then for any $A$-left module of $M$, the localization
$
E(M)[S^{-1}] 
$
of the injective hull $E(M)$ of $M$ 
is isomorphic to the injective hull of $M[S^{-1}]$ as an $A[S^{-1}]$-module.
\end{lemma}

\begin{lemma}
Let $A$ be a ring which is finite over its noetherian central subalgebra $R$.
Let $S$ be a multiplicative subset of $R$.
If $A$ is Auslander regular, then $A[S^{-1}]$ is also Auslander regular.
\end{lemma}

With the above proposition and the lemma,
we may easily prove the following proposition.
\begin{proposition}
\label{proposition:AuslanderRegular}
Let $A$ be an algebra 
which is finite over its Noetherian central subalgebra $R$. 
Let us denote by $A_X$  
the sheaf of $\Oo_{X}$-algebras over $ X=\Spec(R)$ associated to $A$.
Then the following are equivalent:
\begin{enumerate} \item $A$ is an Auslander regular ring.
\item There exists an affine open covering $\{U_i\}$ of $\Spec(R)$ such that
each ring $\Gamma(U_i,  A_X)$  is Auslander regular. 
\item There exist elements  $f_1,f_2,\dots, f_t \in R$ such that their  
zero-locus $V(f_1,\dots, f_t)$ is empty and 
each localization $A[f_j^{-1}]$ is Auslander regular.
\item For any prime ideal $\pe$ of $R$,
 $A\otimes_R R_\pe$ is Auslander regular.
\suspend{enumerate}
If in addition, $R$ is a finitely generated 
commutative algebra over a field $\kk$,
the above conditions are also equivalent to:
\resume{enumerate}
\item For any maximal ideal $\m$ of $R$, 
$A\otimes_R R_\m$ is Auslander regular.
\end{enumerate}
\end{proposition}

\begin{proof}
We have  trivial implications $(1)\implies (2)\implies (3) \implies (4) \implies (5)$.
To prove $(4)\implies (1)$, we take an arbitrary finite $A$-module $M$
and consider $\Ext^\bullet (N,A)$ for $N\subset \Ext^\bullet(M,A)$.
Since these modules are finite, we may easily deduce the result.
Similarly, when $R$ is finitely generated over a field, then we have $(5)\implies (1)$.
\end{proof}

Looking at the proposition above, we may define Auslander regularity of 
sheaf of algebras as follows:
\begin{definition}
\label{definition:AuslanderRegularScheme}
Let $X$ be a scheme. Let $\mathcal A $ be a sheaf of $\Oo_X$-algebras 
such that $\Oo_X$ is contained in the center of $\mathcal A$ and 
$\mathcal A$ is coherent as an $\Oo_X$-module. 
Then $\mathcal A$ is said to be
{\it Auslander regular} if for any point $P$ of $X$,
  the stalk $\mathcal A_P$ of $\mathcal A$ is Auslander regular.
\end{definition}
We see easily from Proposition
\ref{proposition:AuslanderRegular} the following.
\begin{proposition}
\label{proposition:AuslanderRegularScheme}
Let $\mathcal A$  be a sheaf of algebras on a scheme $X$ such that
the assumptions of Definition 
\ref{definition:AuslanderRegularScheme} are satisfied. 
Then the following conditions are equivalent:
\begin{enumerate}
\item $\mathcal A$ is Auslander regular.
\item For each affine open subset $U$ of $X$,  the algebra $\mathcal A(U)$ 
of sections on $U$ is Auslander regular.
\item There exists an affine 
open covering $\{U_\lambda\}_{\lambda \in \Lambda}$ of $X$ such that 
each $\mathcal A(U_\lambda) $ is Auslander regular.
\suspend{enumerate}
If in addition, $X$ is a scheme of finite type over a field $\kk$,
then the above conditions are also equivalent to:
\resume{enumerate}
\item For any closed point $\m$ of $R$, 
the stalk $A_\m$ is Auslander regular.
\end{enumerate}

\end{proposition}
\qed


\begin{proposition}
Let $A$ be a ring which is finite over its central noetherian subring $R$.
We assume $A$ has formally completely indecomposable fiber at 
each closed point $\m \in \Spm(R)$.
Then the algebra $A$ is Auslander regular 
if and only if its completion $\hat A_{\M}$ is Auslander 
regular for any maximal ideal $\M$ of $A$.
\end{proposition}

\begin{proof}

In view of the previous Proposition \ref{proposition:AuslanderRegularScheme},
we see that $A$ is Auslander regular if and only if each fiber $A_\m$ is
Auslander regular at each closed point $\m\in \Spm(R)$.
Since $\hat R_\m$ is flat over $R_\m$, the Auslander regularity of $A_\m$
is the same as that of $A\otimes_R \hat R_\m$.
By the complete decomposability, we see that 
the property is also equivalent to the
regularity of $\hat A_{\M}$ for each maximal ideal $\M$ of $A$ 
which lies over $\m$.
\end{proof}

%
%
%


\section{Regularity of norm based extensions of Weyl algebras}
In this section we describe the main result of this paper.
We first show that the norm based extension of the Weyl algebra is 
Auslander regular. We then show that the Auslander regularity distinguish 
the Weyl algebra $A$ with other algebras $B$
 which are Zariski locally isomorphic to $A$. 

\subsection{Norm based extensions of Weyl algebras}
\subsubsection{definition of Norm based extensions of Weyl algebras}
In \cite{T9}, the author defined the ``norm based extension''
of a Weyl algebra $A$. Let us make a short review.

For the rest of the paper we denote by $A$ the Weyl algebra $A=A_n(\kk)$ 
over a field $\kk$ of positive characteristic. 
It is an algebra generated by 
elements $\{\gamma_i\}_{i=1}^{2n}$  which 
satisfy the following commutation relations,
$$
[\gamma_i, \gamma_j]=h_{i j} \quad(i,j=1,2,\dots,2n)
$$
where $(h_{i j})_{i,j=1}^n$ is a skew symmetric non degenerate constant matrix.
Such generators are referred to as CCR generators.
We denote by $X=\Spec(R)$ the affine spectrum of the center $R=Z(A)$ of
$A$. $X$ carries a sheaf of $\Oo_X$-algebras corresponding to $A$, which we
denote by $A_X$.

Now, using the set $\{\gamma_i\}_{i=1}^{2n} $ of CCR generators,
we notice that the center $Z(A)$ of $A$ has ``linear coordinates''
 $\{\gamma_i^p\}_{i=1}^{2n}$
which gives rise to an identification
$$
X=\Spec Z(A)
 \isoto \A^{2n}.
$$
We may then consider a compactification of $X$ by putting 
$$
 \bar X =\mathbb P^{2n}.  
$$
It should be noted that the compactification $\bar X\supset X$ 
depends on the choice of the CCR generators of $A$.

Let us extend the sheaf $A_X$ on $X$ to $\bar X$. 
To do that, we notice that the algebra $A$ admits a norm $N_A: A\to Z(A)$ 
(or, to call it more appropriately in accordance with
 the general theory of 
Azumaya algebra, a ``reduced norm'') which is essentially given by 
considering determinants on fibers.  It gives rise to a morphism
$$
N_A: A_{X} \to \Oo_X.
$$
of sheaves on $X$. 
Noting that $A$ is ``asymptotically commutative'' near the boundary 
of $X$ and that $N_A$ 
gives the asymptotic behavior of elements of $A$, we are able
to define an extension of the sheaf $A_X$  on $X$ to the whole of
$\bar X$:
\begin{definition}\cite[Definition2.3]{T9}
The {\it norm based extension} $A_{\bar X}$ of the Weyl algebra $A$ is 
an  $\Oo_{\bar X}$-algebra on the projective space $\bar X$
whose sections on an open set $U$ of $\bar X$ are given as follows:
$$
A_{\bar X}(U)=\{s \in A_{X}(U\cap X); N_A (s)
\text{ extends to an element of }
 \Oo_X(U)\}.
$$
\end{definition}

It turns out that $A_{\bar X}$ is a sheaf of $\Oo_{\bar X}$-algebras. 

\subsubsection{Coordinates of the norm based extension at infinity}
To describe $A_{\bar X}$ at a point $\m$ on the boundary $H=\bar X \setminus X$ 
more closely,
let us describe it in terms of coordinates.
It is convenient to introduce  (non commutative) ``projective coordinates''.
Namely, we introduce  variables
$\Gamma_1,\dots, \Gamma_{2 n},Z$
such that:
\begin{enumerate}
\item  $Z$ is central.
\item  $\gamma_i=\Gamma_i/Z$ ($i=1,2,\dots {2n}$).
\item 
$
[\Gamma_i,\Gamma_j]=h_{i j}Z^2 
\qquad(i,j=1,2,\dots, 2n).
$
\end{enumerate}
By making a linear change of coordinates if necessary,
we may assume without loss of generality that 
$\Gamma_1\neq 0$ at $\m$.
We then employ a local coordinate system 
$(v,u, \{\bar\gamma_i\}_{i=3}^{2 n})$
around $\m$ defined as follows.
$$
v=- \Gamma_2   \Gamma_1^{-1} , 
\quad
u=Z \Gamma_1^{-1} ,
\quad
\bar\gamma_i= \Gamma_i \Gamma_1^{-1}  \quad (i=3,4,\dots, 2n). 
$$
We may also assume that $h_{i j}$ is in a ``normal form'' such that 
the coordinates are subject to the following commutation relations.
\begin{align*}
[v,u]=u^3&,
\quad
[u,\bar\gamma_i]=0,
\quad
[v,\bar\gamma_i]=u^2 \bar\gamma_i,
\quad
[\bar\gamma_i,\bar\gamma_j]=h_{i j} u^2\\
&(i,j=3,4,\dots,2 n).
\end{align*}

To sum up, we have proved the following:
\begin{proposition}\label{proposition:Agenerator}
Let $U$ be an affine piece of the projective space $\bar X$. 
Then the algebra of sections $A_{\bar X}(U)$ of $A_{\bar X}$ on $U$ 
is generated by $2 n$ generators.
Namely, we have $A_{\bar X}(U)=
\kk\langle u,v,\bar\gamma_3,\bar\gamma_4,\dots, \bar\gamma_{2n}
\rangle$,
where the generators are subject to the following relations.
\begin{align*}
[v,u]=u^3&,
\quad
[u,\bar\gamma_i]=0,
\quad
[v,\bar\gamma_i]=u^2 \bar\gamma_i,
\quad
[\bar\gamma_i,\bar\gamma_j]=h_{i j} u^2\\
&(i,j=3,4,\dots,2 n).
\end{align*}
Using these generators, $\Oo_{\bar X}(U)$ is expressed as a polynomial 
algebra  $\Oo_{\bar X}(U)=\kk[v^p,u^p,\bar\gamma_3^p,\bar\gamma_4^p,\dots,\bar\gamma_{2n}^p]$.
The defining equation of the hyperplane $H$ at infinity in $U$ 
is given by $u^p$.

\end{proposition}
The last two sentences are also checked easily.
\qed

\subsubsection{Regularity of norm based extensions of Weyl algebras}
We now state a main result of this paper:
\begin{proposition}\label{proposition:main}
The norm based extension $A_{\bar X}$ of the Weyl algebra $A$ 
is 
Auslander regular. 
\end{proposition}
\begin{proof}
As we have seen in 
Proposition \ref{proposition:AuslanderRegular}, 
it is enough to 
prove that the stalk $A_{\bar X,\m}$ of the sheaf $A_{\bar S}$ at 
each closed point $\m$ of ${\bar X}$ is an Auslander regular ring.
Let us first deal with the case where $\m\in X$. $\m$ is then identified 
with the maximal ideal of $R=Z(A)$.
Looking at Proposition \ref{proposition:AuslanderRegular} once again, 
we see that our claim in this case is  
a consequence of the fact that $A$ is a regular ring which is 
 already well-known.
Indeed, since $A=A_n(\kk)$ has a filtration whose associated
graded algebra is a polynomial ring, 
(putting aside of verifying the closure condition of
such filtration,) we see by using the theorem of 
Bj\"ork(Theorem\ref{theorem:Bjork} that $A_n(\kk)$ 
is Auslander regular.  
Here, however, let us prove directly 
that $A_{\bar X,\m}=A\otimes_R R_{\m}$ is Auslander regular.
By using a structure theory of the Weyl algebras (such as \cite[Lemma 3]{T4}),
we see  that
$\gr_{\m}A_{\m}$ is isomorphic to the full matrix algebra
 $M_{p^n}(R/\m)[t_1,\dots,t_{2n}]$ over a polynomial ring.
Thus by using Corollary \ref{corollary:auslanderregular} we see that 
$\gr_{\m } A_{\m}$ is Auslander regular and hence 
by using the theorem of Bj\"ork
 we conclude  that  $A_{\m}$ is Auslander regular.

Let us next deal with the case where 
 $\m$ is on the hyper plane $H=\bar X \setminus X$ 
at infinity. 
In terms of the coordinate described in Proposition \ref{proposition:Agenerator},
 the stalk $A_{\bar X,{\m}}$ is a quasi local ring with the maximal ideal
 $\M$ generated $\{u,v\} \cup \{\bar\gamma_i \}_{i=3}^{2n}$.
The associated graded algebra $\gr_\M (A_{\bar X, \m})$ 
is isomorphic to an algebra $\tilde A$
generated by elements 
$\{\tilde u,\tilde v\}\cup \{\tilde \gamma_i\}_{i=3}^{2n}$ 
which satisfy the following relations.
$$
[\tilde v,\tilde u]=0,
\quad
[\tilde\gamma_i, \tilde \gamma_j]=h_{i j} \tilde u^2,
\quad
[\tilde u,\tilde \gamma_i]=0,
\quad
[\tilde v,\tilde \gamma_i]=0.
$$
The algebra $\tilde A$ again has an filtration by $u \tilde A$ whose associated
graded ring $\gr_{u \tilde A }\tilde A$ is a usual commutative polynomial ring.
By the theorem of Bj\"ork (Theorem \ref{theorem:Bjork}), 
we see that the algebra $\tilde A$ is Auslander regular 
and then by the same theorem we see that $A$ is also Auslander regular.

\end{proof}

\subsection{Inverse Frobenius pull-backs}
Before proceeding further, we need to review the definition of 
``inverse Frobenius pull-backs'' in \cite{T9}.
In short, $X^\dagger$ is obtained by adjoining $p$-th power roots of the 
coordinate functions of $X$.

In order to explain the notation more clearly, it might be helpful to
compare it with a notation used in another paper.
Let us compare here our notation with the one which appears in a 
paper of Illusie\cite{illusie1}.
There one sees at (2.1.1) a notation such as $X^{(p)}$. 
In short, ``$X^\dagger $ to $X$  in our paper is what $X$ is to $X^{(p)}$
in the paper of Illusie''.
To avoid confusion, let us write $Y$
in place of the  $X$ in the Illusie's paper.
The diagram (2.1.1) then looks like this.
$$
\xymatrix{
Y\ar[d]  & \ar[l] Y^{(p)} \ar[d] &\ar[l]_{F_{Y/S}} 
\ar[dl] \ar@(u,u)[ll]_{F_Y}Y  \\
S &\ar[l]_{F_S} S&
}
$$
In the present paper, we concentrate on a case where $S=\Spec(\kk)$.
Then the diagram above may be rewritten  
in the present paper's notation
provided we put $X^\dagger=Y$:
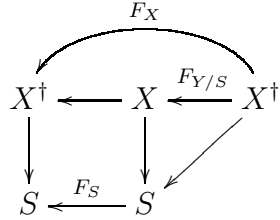
$$
\xymatrix{
X^{\dagger}\ar[d]&\ar[l] X \ar[d]&
 \ar[l]_{F_{Y/S}} \ar[dl]  \ar@(u,u)[ll]_{F_X} X^\dagger \\
S& \ar[l]_{F_S} S&
}
$$

The same sort of idea applies and our spaces  
$X,\bar X$, and  $H$ has their own inverse Frobenius pull-backs
who are denoted by 
$X^\dagger,\bar X^\dagger$, and  $H^\dagger$.

\subsection{$\Oo_{\bar X^\dagger}$ as a classical counterpart of $A_{\bar X}$}
Since the inverse Frobenius pullback $\bar X^\dagger$ and $\bar X$ have the
same underlying space $|\bar X^\dagger|=|\bar X|$, we identify 
$\Oo_{\bar X^\dagger}$ with a sheaf of $\Oo_{\bar X}$-algebra on $X$.
$\Oo_{\bar X^\dagger}$ is thought of as a 
``classical counterpart'' of our sheaf $A_{\bar X}$.
Indeed, on the affine piece $X$ of $\bar X$,
$\Oo_{\bar X^\dagger}$ is generated by 
$t_1=(\gamma_1^p)^{1/p},\dots,t_{2n}= (\gamma_{2n}^p)^{1/p}$.
The affine coordinate ring $\kk[t_1,t_2,\dots, t_{2n}]$ resembles
the Weyl algebra $A_n(\kk)$ in many ways.
For example,  
Both algebras are free module of rank $p^{2n}$ over 
$Z=\Spec(\kk[t_1^p,\dots,  t_{2n}^p]$.
Moreover, there exists a filtration on $A_n(\kk))$ such that
the associated graded algebra $\gr(A_n(\kk))$ is isomorphic to
the polynomial ring $\kk[t_1,t_2,\dots, t_{2n}]$.
We may say that $A_n(\kk)$ is `asymptotically commutative' and 
gradually looks like the polynomial ring $\kk[t_1,t_2,\dots, t_{2n}]$.

Let us look at this  by studying the behavior of the sheaf $A_{\bar X}$
around the hyperplane $H$ at infinity.
We have shown in \cite{T9} that there exists a specific ideal $\Jot_{H^\dagger}$
of $A_{\bar X}$. We may regard it as a non commutative analogue of the ideal
sheaf of $H^\dagger$.
It is related to the ideal sheaf  $\Iscr_H$ of $H$ by the following equation:
$$
\Jot_{H^\dagger}^p\isoto \Iscr_H A_{\bar X}.
$$
The situation may probably be well-understood in terms of coordinates
introduced in
Proposition \ref{proposition:Agenerator}.
$\Iscr_H$ is an ideal of $\Oo_{\bar X}$ generated by $u^p$, whereas
$\Jot_{H^\dagger}$ is an ideal of $A_{\bar X}$ generated by $u$.
%
The quotient $A_{\bar X}/\Jot_{H^\dagger} $ is isomorphic to 
$\Oo_{H^\dagger}$.
As a result, we obtain a ``symbol map''
$$
\rho:A_{\bar X} \to \Oo_{H^\dagger}
$$ 
which essentially sends each element of $A_{\bar X}$ to its principal symbol
restricted to $H^\dagger$.
For the ease of future reference, let us write the fact down in a following:
\begin{proposition} \label{proposition:symbol}
There exists a ``symbol map''
$$
\rho:A_{\bar X} \to \Oo_{H^\dagger}
$$ 
which gives rise to an isomorphism
$$
\bar\rho:A_{\bar X} /\Jot_{H^\dagger}\to \Oo_{H^\dagger}.
$$ 

The following diagram  is commutative.
$$
\xymatrix{
&A_{\bar X} \ar[r]_{\rho} \ar[d]_{N_A} & \Oo_{H^\dagger}\ar[d]^{\bullet^{p^n}} \\
& \Oo_{\bar X^\dagger}  \ar[r]^{\text{restr.}}   &  \Oo_{H^\dagger}
}
$$
(Note however that $N_A$ is not additive.
\end{proposition}
\qed

\section{A characterization of the Weyl algebra among its cousins}
\subsection{Algebras Zariski isomorphic to a Weyl algebra}

We continue to consider a Weyl algebra $A$ 
 over a field of positive characteristic. 
and employ $X, \bar X, H=\bar X \setminus X$ as in the previous subsection.
Let us study  a ``cousin'' of $A$.
Namely, we consider an $Z$-algebra $B$ such that its associated sheaf
$B_X$ is a sheaf of algebra which is Zariski locally isomorphic to $A_X$
as an $\Oo_X$-algebra. By abuse of language we call such $B$
 an  algebra Zariski locally isomorphic to $A$.

As described for example in \cite{Kontsevich1},\cite{T7},
the algebra $B$ is in one to one correspondence with an $A$-module 
$W$ of rank one whose associated sheaf $W_X$ on $X$ 
is a locally free $\Oo_X$-module.
The correspondence is given by the following relation.
$$
B=\End_A (W)
$$

It is known that we have a bunch of 
non trivial locally free $A$-modules of rank $1$. 
 Indeed, when $\kk$ is
a field of characteristic $0$, projective $A_1(\kk)$-modules of rank $1$ 
are parametrized by ``Calogero-Moser spaces''(\cite{BC}).
We may then restrict the coefficient ring and consider them 
''mod $p$'' and obtain such modules for our case.)
  So $A$ has a lot of such ``cousins''.

We also note that a reflexive $A$-module $W$ of rank one admits a
norm $N_W: W \to Z$ and
its associated sheaf $W_X$ has a norm based extension $W_{\bar X}$,
 a sheaf on $\bar X$ which extends $W_X$ to $\bar X$
(See \cite{T9}.)
Likewise, we see easily that an algebra $B$ 
which is locally isomorphic to $A$ admits a norm $N_B: B\to Z$
and that 
its associated sheaf $B_X$ has a norm based extension $B_{\bar X}$ 
as a sheaf of $\Oo_{\bar X}$-algebras in a similar manner.

The sheaf $B_{\bar X}$ is of course locally isomorphic to $A_{\bar X}$ on
$X=\A^{2n}$. However, It is not always true that $B_{\bar X}$ is
locally isomorphic to $A_{\bar X}$ on the whole of $\bar X=\mathbb P^{2n}$.
The truth is far from that. In fact, $B_{\bar X}$ is 
Auslander regular only in the trivial case where $B$ is isomorphic to $A$
as we will prove later.

The following lemma is a consequence of the Auslander regularity of 
$A_{\bar X}$ and is a generalization of 
the equivalence (1) $\iff$ (4) of \cite[Theorem 2.4]{T8}

\begin{lemma}
Let $W$ be a reflexive $A$-module of rank one.
Assume $W_{\bar X}$ is a
locally free $\Oo_{\bar X}$-module on a Zariski open subset $U$ of
$\bar X$. Then $W_{\bar X}$ is $A_{\bar X}$-locally free on $U$.
\end{lemma}
\begin{proof}
Let $x$ be a closed point of $U$. Since the dimension of $X$ is 
greater or equal to $2$, we have $\dim \Oo_{\bar X,x} \geq 2$.
Since $W_{\bar X}$ is locally free on $U$, we have
$\depth_{\Oo_{\bar X,x}} W_{X,x}\geq 2$.
Since $A_{\bar X}$ is regular, we see by definition 
that every stalk of $A_{\bar X}$ has
a finite global dimension. In particular,
 the projective dimension
$\pd_{ A_{\bar X, x}} W_{\bar X,x}$ is finite and according to
Ramras\cite[Lemma 2.9]{Ramras1}, it is equal to
the following quantity which is, under the notation of Ramras,
denoted as ``$\Ext_{A_{\bar X,x}}\text{-}\dim W_{\bar X, x}$''.
$$
\sup\{j \geq 0; 
\Ext^j_{ A_{\bar X,x}}
(\dim W_{\bar X, x} ,
A_{\bar X,x})
\neq 0
  \}.
$$
According to 
another Proposition of Ramras\cite[proposition 2.14]{Ramras1}, 
the quantity is equal to $0$.
Thus the projective dimension of  $W_x$ is equal to zero and  hence 
$W_x$ is locally free at $x$.

\end{proof}

\begin{corollary}\label{corollary:ae}
There exists a closed subset $F$ of codimension at least $3$ of $\bar X$
such that $W_{\bar X}$ is $A_{\bar X}$-locally free on $V_0=\bar X \setminus F$.
Accordingly, we have:
\begin{enumerate}
\item $A_{\bar X}$ is locally isomorphic to $B_{\bar X}$ on
$\bar X\setminus F$.
\item $B_{\bar X}$ is a reflexive extension of $B_{\bar X}|_{V_0}$.
\end{enumerate}

\end{corollary}

\qed

\subsection{Summary of results on reflexive sheaves}
In the sequel, we need to handle reflexive sheaves. Let us briefly
summarize here some of the results on reflexive sheaves which appear in
the usual commutative algebraic geometry..

Recall that a coherent sheaf $\Eff$ on a scheme $X$ is reflexive
if it is isomorphic to its double dual.

\begin{proposition}\cite[Proposition 1.1]{Ha2}
A coherent sheaf $\Eff$ on a noetherian integral scheme $X$ is reflexive if
and only if (at least locally) it can be included din an exact sequence
$$0 \to \Eff \to \Ee \to \Ge\to 0,$$
where $\Ee$ is locally free and $\Ge$ is torsion-free.

\end{proposition}

\begin{proposition}\cite[Proposition 1.9]{Ha2}
\label{proposition:rank_one_reflexive}
Assume that the base scheme $X$ is integral and locally factorial. 
Then any reflexive sheaf of rank one on $X$ is invertible.
\end{proposition}

Recall that a coherent sheaf $\Eff$ on  $X$ is normal if for every open
set $U\subseteq X$ and every closed subset $Y\subseteq U$ of 
codimension $\geq 2$, the restriction map 
$\Eff(U) \to \Eff(U\setminus Y)$ is bijective.

\begin{proposition}\cite[Proposition 1.6]{Ha2}
Let $\Eff$ be a coherent sheaf on a normal integral scheme $X$. 
The following conditions are equivalent:
\begin{enumerate}
\renewcommand{\labelenumi}{({\roman{enumi}})}
\item $\Eff$ is reflexive.
\item $\Eff$ is torsion-free and normal.
\item $\Eff$ is torsion free, and for each open $U\subseteq X$ and
closed subset $Y\subseteq U$ of codimension $\geq 2$, 
$\Eff_U\isoto j_*\Eff_{U\setminus Y}$, where $j:U\setminus Y\to U$ is the 
inclusion map.
\end{enumerate}

\end{proposition}

Let us now cite here the following theorem which is a generalization of 
the Horrocks\cite{Horrocks1}. 
Although we only use it in an original assumption
of Horrocks, it gave a good guideline for our study. 
It is also used in the study reflexive sheaves 
in the author's another paper\cite{T9}.
\begin{theorem}
\cite[Theorem 0.2]{AbeYoshinaga} \label{Theorem:abeyoshinaga}
Let $\kk$ be an algebraically closed field, 
$n$ be an integer greater than or equal to $3$,
 and let $E$ be a reflexive sheaf on $\mathbb P^n_{\kk}$
 of rank $r (\geq 1)$. Then $E$ splits into
a direct sum of line bundles if and only if there exists a hyperplane 
$H\subset P^n_{\kk}$
such that $E|_H$ splits into a direct sum of line bundles.
\end{theorem}

In the proof of the above-cited theorem, Abe and Yoshinaga shows the following
theorem, which is also useful for us.

\begin{theorem}(\cite[Theorem 2.2]{AbeYoshinaga})\label{Theorem:splitting}
Let $E$ and $F$ be reflexive sheaves on $\mathbb P^n_{\kk}$
$(n \geq 2)$ and $H$ be a hyperplane in
$\mathbb P^n_{\kk}$.
 Suppose $E|_H \isoto F|_H$ and 
$\Ext^1_{\mathbb P^n_{\kk}}(F,E(-1))=0$
Then $E  \isoto F$.
As is shown in the proof of \cite[Theorem 2.2]{AbeYoshinaga},
 for each isomorphism $\psi:E|_H \to F|_H$, 
there exists its extension $\Psi:E\to F$ which extends $\psi$.
\end{theorem}

\subsection{Serre twist of $B_{\bar X}$}

As we have seen,   there exists a specific ideal $\Jot_{H^\dagger}$
of $A_{\bar X}$. 
By taking suitable tensor power of $\Jot_{H^\dagger}$, 
we may consider ``Serre twist''
$\Jot_{H^\dagger}^{c}$ 
of $A_{\bar X}$ 
described  at \cite[lemma 2.6]{T9}.  
Likewise, there is also a ``Serre twist'' of the sheaf $B_{\bar X}$.
To emphasize the similarity with Serre twist,
we use notation $\Serre{B}{k}$ as the ``$B$-version'' 
of $\Jot_{H^\dagger}^{-k}$. (Note the sign convention which is used
in order to fit the notation of the usual Serre twist.) 
To be more precise, we define as follows.
\begin{definition}\label{definition:SerreTwist} 
For any $c\in \Z$ we define a presheaf
$\Serre{B}{k}$ on $\bar X$ by
$$
\Serre{B}{k}(U)
=\{
s \in B_{X}(U\cap X) ;  \ \ord_{H^\dagger}(N_B(s))\geq -k p^n
\},
$$
where $\ord_{H^\dagger}$ denotes the  valuation with respect to
the prime divisor $H^\dagger$.
For $k=-1$, we also use the notation
$ \Jot_{B,H^\dagger}$ for $\Serre{B}{-1}$. 
\end{definition}
When $B=A$,
$ \Jot_{B,H^\dagger}$ is identical with the $\Jot_{H^\dagger}$
defined as above.

\begin{lemma}
\label{lemma:SerreTwist}
The following facts are true.
\begin{enumerate}
\item 
For any integer $k$, 
the presheaf $\Serre{B}{k}$
 defined in Definition \ref{definition:SerreTwist}  is actually  a sheaf.
\item  
For any integer $k$, 
the sheaf $\Serre{B}{k}$
 is an $B_{\bar X}$-{\it bimodule}.
\item  
For any integer $k$, 
the sheaf $\Serre{B}{k}$
is $\Oo_{\bar X}$-reflexive.
\item 
$\Jot_{B,H^\dagger}$ is an ideal sheaf of $B_{\bar X}$. Moreover,
for any positive integer $k$, we have a natural inclusion
$$
\Serre{B}{-k}
\supset (\Jot_{B,H^\dagger} )^k
$$
which is an isomorphism outside of a closed set of codimension at least $2$ on $\bar X$. 
\item We have
$$
\Serre{B}{-p}
= \Iscr_{H} B_{\bar X}
$$

\end{enumerate}

\end{lemma}
\begin{proof}
(1),(2),(4)  are easily verified by the definition.

\noindent
(3): It is also easily to check that $\Serre{B}{k}$ is normal and torsion free, 

\noindent (5):
The right hand side is locally isomorphic to $B_{\bar X}$ as an $\Oo_{X}$-module
 and is in particular a reflexive $\Oo_{\bar X}$-module. 
Since both hand sides are reflexive and are equal on outside of a
closed subset of codimension at least $2$, they are equal.
\end{proof}

As a result of the lemma above, we have a
descending chain $\{\Serre{B}{k}\}_{k\in \Z}$
of $B_{\bar X}$-bimodules. The associated graded module plays
an important role in this paper.

\subsection{Formally indecomposability of $B_{\bar X}$.}

There exists a symbol map
$$
\bar \rho_B: 
B_{\bar{X}}\to \Oo_{H^\dagger}.
$$
which is defined as follows:  there exists an open subset $V_0$ of $\bar X$
where  $B_{\bar X}$ is locally isomorphic to $A_{\bar X}$ . 
Then on the open set $V_0$, we may define $\rho_B$ using the map 
$\rho$ in Proposition \ref{proposition:symbol} via
such local isomorphisms. Note that the consistency of $\rho$ and the norm map 
assures us that $\rho_B$ so defined is independent of the choice of such 
local isomorphism.
We may then extend $\rho_B$ to the whole of $\bar X$
by using the normality.
The kernel of $\rho_{B}$ is equal to $\Jot_{B,H^\dagger}$. This fact is also 
verified at first on $V_0$ and then verified on the whole of $\bar X$ by
using the Hartogs property.

\begin{definition}
We put $\bar B_H=B_{\bar X}/\Jot_{B,H^\dagger}$.
\end{definition}

There exists a symbol map
$$
\bar \rho_B: 
B_{\bar{X}}\to \Oo_{H^\dagger}.
$$

\begin{proposition}\label{proposition:BOH}Following facts are true.
\begin{enumerate}
\item $
\bar B_H=B_{\bar{X}}/ \Jot_{B,H^\dagger} \subset \Oo_{H^\dagger}
$
\item \label{lemma:refextofbarB}
The reflexive hull of $\bar B_H$ as an $\Oo_{H}$-module is isomorphic to  $\Oo_{H^\dagger}$.
\end{enumerate}
\end{proposition}
\begin{proof}
(1): We consider the symbol map
$$
\bar \rho_B: 
B_{\bar{X}}\to \Oo_{H^\dagger}
$$
and we claim
$$
\Ker(\bar \rho_B)=\Jot_{B,H^\dagger}.
$$
Indeed, by using the local generators we may easily verify that 
the both hand sides are equal on $V_0$. 
Since they are reflexive $\Oo_X$-modules, they
are equal on the whole of $\bar{X}$.

\noindent (2): It goes without saying that $\Oo_{H^\dagger}$ is locally free $\Oo_{H}$ module. In particular, it is 
reflexive. On the other hand, the two sheaves $\Oo_{H^\dagger}$ and 
$\bar B_H$ coincide except at a locus of codimension $3$ in $\bar X$.

\end{proof}

\begin{proposition}\label{proposition:indecomposability}
For any closed point $\m$ of $\bar X$, 
There exists only one maximal ideal
of $B_{\bar X}$ which lies over $\m$.
\end{proposition}
\begin{proof}
When $\m$ is in the affine piece $X$, the result is a consequence of 
the structure theory of the sheaf of algebras $A_X$ on $X$ which 
is essentially well known (see for example \cite{T4}).
Let us treat the case where $\m$ is on the hyperplane $H$ at infinity.
Let $\M$ be a maximal ideal of $B_{\bar X}$ which lies over $\m$.
Since $\m$ is on $H$, we have $ \m\supset \Iscr_{H}A $ so that we have
$
\M \supset \Iscr_H B_{\bar X}.
$
By using Lemma \ref{lemma:SerreTwist}, we have 
$$
\Jot_{B,H^\dagger}^p
\subset
\Serre{B}{-p}
= \Iscr_{ H} B_{\bar X}
\subset \M.
$$
By using the maximality of $\M$, we easily see that
$\Jot_{B,H^\dagger}
\subset \M.
$
(otherwise we would have 
$\Jot_{B,H^\dagger}+\M=B_{\bar X}$
which implies in particular 
that $1$ is nilpotent in $B_{\bar X}$, which is absurd.)
Thus $\M$ corresponds to a maximal ideal 
$
\bar \M=
\M/
\Jot_{B,H^\dagger}$ of 
$\bar B=B_{\bar X}/ \Jot_{B,H^\dagger}$. 
As stated in Proposition \ref{proposition:BOH},
$\bar B$ is a subsheaf of $\Oo_{H^\dagger}$.  
Let us put
$$
\m^\dagger
\ \defeq
\{ f\in \Oo_{H^\dagger}| f^p \in \m\}
=\Ker(\Oo_{H^\dagger} \overset{\operatorname{Frob}}{\to} \Oo_H\to \Oo_H/\m).
$$
We may easily see from the definition that 
we have  $
\m^\dagger \cap \bar B \subset \bar \M
$ and so we have the following inclusion of commutative domains.
$$
\Oo_H/\m
\subset 
\bar B/(\m^\dagger \cap \bar B) 
\subset
 \Oo_{H^\dagger} /\bar \M.
$$
 Note that $\Oo_{H^\dagger} /\bar \M $ is a finite extension field of
the field
$\Oo_H/\m$. 
We conclude that the middle side of the above inclusion is also a field,
and that $\m ^\dagger \cap \bar B$ is a maximal ideal of $\bar B$.
We thus have $\m^\dagger \cap \bar B=\bar \M$. 

\end{proof}

\subsection{Behavior of $B_{\bar X}$ on loci where $B_{\bar X}$ is 
Auslander regular}
%
%
The following lemma, which is essentially a consequence of the theorem of Ramras
(Theorem \ref{theorem:Ramras}),
is an central heart of our next result. 
\begin{lemma}\label{lemma:main}
Let $B$ be a ring Zariski locally isomorphic to a Weyl algebra
$A=A_n(\kk)$ over a field of positive characteristic $p$.
We equip $B_{\bar X}$ with the filtration $\{\Serre{B}{-k}\}_{k =0}^\infty$. 
Let $U$ be a Zariski open subset of $\bar X$.
We assume $B_{\bar X}$ is Auslander regular on $U$.
Then the following facts are true on $U$:
\begin{enumerate}
\item $B_{\bar X}$ is locally free as an $\Oo_{\bar X}$-module.
\item  $ \gr_0(B/\Iscr_{H} B)=\bar B_{H}=B/\Jot_{H^\dagger}$ is isomorphic to 
$\Oo_{H^\dagger}$ as 
an  $\Oo_{H}$-algebra.
\item For each $k$, $\gr_k (B/\Iscr_{H} B)$ is a locally free 
$\Oo_{H^\dagger}$-module of rank one.
\end{enumerate}
\end{lemma}
\begin{proof}
(1):
Let $\m$ be a closed point of $U$. 
As we proved in Proposition \ref{proposition:indecomposability}, 
there exists only one maximal ideal $\M$ of $B_{\bar X}$ which
lies over $\m$.
The $\M$-adic completion 
$\hat B_{\bar{X},\M}=\varprojlim_{j} B_{\bar X}/\M^j$ of $B_{\bar X}$ is isomorphic to 
$B_{\bar X,\m} \otimes_{\Oo_{\bar X,\m}} \hat \Oo_{\bar X,\m}$.
It is a ring which is finite over a central 
local ring $\hat \Oo_{\bar X,\m}$.
By the theorem of Ramras (Theorem \ref{theorem:Ramras}),
 we see that 
$\hat B_{X,\M}$ is free over $\hat R_{\m}$.
 Thus $B_{\bar X}$ is 
locally free on $U$.

Let us prove (2). This is done in several steps.

Let us first take a look at a sheaf
$ \mathcal M= \Serre{B}{-(p-1)}/\Iscr_H B_{\bar X}$.
We  note that  
the norm map $N_B: B_{\bar X} \to \Oo_{\bar X}$ factors through
a map
$$
\bar N_B: B_{\bar X}/\Iscr_H B_{\bar X} \to 
\Oo_{\bar X}/\Iscr_{H}^{p^n}
$$
so that the sheaf $\mathcal M$ may be rewritten as follows.
$$
\mathcal M=\{s \in B_{\bar X}/\Iscr_H B_{\bar X} 
\ |\  \ord_{H^\dagger}(\bar N_{B}(s))\geq (p-1)p^n\}
$$ 
Since the sheaf $B_{\bar X}/\Iscr_H B_{\bar X}$ is locally free,
by checking the normality and torsion freeness, 
we see that $\mathcal M$ is a reflexive $\Oo_{\bar X}$ module.


Let us then focus on the action $\alpha$ 
of $\bar B_{H}=B_{\bar X}/\Jot_{H^\dagger}$ 
on the sheaf $\mathcal M$:
$$
\alpha:\bar B_{H} \otimes \mathcal M \to \mathcal M
$$
Note that the reflexive extension of 
$\bar B_{H}$
 as a $\Oo_{H}$-module is equal to  $\Oo_{H^\dagger}$ 
(Proposition \ref{proposition:BOH}).
For any $f\in \Oo_{H^\dagger}(U\cap H)$, the action $\alpha$ gives 
a local section $\alpha(f)$ outside of a closed subset of $U\cap H$
 of codimension at least $2$. 
Then by virtue of the $\Oo_{H}$-reflexivity of $\mathcal M$, $\alpha(f)$
extends to the whole of $U\cap H$. 
In this way, We may extend the action $\alpha$ 
 to an action of $\Oo_{H^\dagger}$ on $\mathcal M$.
So $\mathcal M$ is now an  $\Oo_{H^\dagger}$ module.
Knowing that $\mathcal M$ is normal and torsion free as an $\Oo_H$-module, 
(and that these
conditions are irrelevant of whether we regard it as $\Oo_H$-module or
 $\Oo_{H^\dagger}$-module,) we see  that
$\mathcal M$ is a reflexive $\Oo_{H^\dagger}$-module of rank one.
As we reviewed in Proposition \ref{proposition:rank_one_reflexive},
this implies that the sheaf 
 $\mathcal M$  is a locally free $\Oo_{H^\dagger}$-module.

Let us now take a point $P$ on $H$ and argue locally around $P$. There exists an affine open subset $V$ of 
$U$ and a section $s_1 \in B_{\bar X}(U)$ such that its residue $\bar s_1 $
generates $\mathcal M$ as an $\Oo_{H^\dagger}$-module. 
By shrinking $V$ if necessary, we may assume that 
the section $s_1$ generates $\Serre{B}{-(p-1)}$ on $V$.
The norm $N(\m_1)$ restricted to $V$ has zeros only on $H$. 
We may also assume that $\Iscr_{H}$ is generated by a single element $z$ 
(``a local defining function of $H$) on $V$.
Let us consider a map $\varphi$ defined as follows.
$$
\varphi: \bar B_{H}\ni f \mapsto f .s_1 \in  \mathcal M.
$$
Then the kernel of $\varphi$ is equal to $\Jot_{H^\dagger}/\Iscr_H B_{\bar X}$.
(This again is firstly trivially verified outside a locus  of 
codimension greater than $2$ and then we may extend the fact to the whole of $H$
by using the reflexivity of $\Jot_{H^\dagger}$.)   
Let us now prove that $\varphi$ is surjective.
$N(s_1)/z^{p^{n-1}(p-1)}$ is invertible on $V$.
Thus 
$N(s_1)/z^{p^{n-1}(p-1)}$ is  invertible  
on a neighborhood $V_0$ of $P$.
Let $s \in \mathcal M_P$ be an arbitrary element of the stalk at $P$.
There exists $m_s \in B_{\bar X,P}$ such that $s=m_s \mod \Iscr_H$.
By definition, $ m_s$ is a germ of a section (also denoted by $m_s$)
of $B_{\bar X}$ on a Zariski open neighborhood $V_1$ of $P$.
Let us then consider a section
$$
b_s=m_s s_1^{-1}.
$$
It is regular on $\complement H \cap V_0 \cap V_1$. By looking at norms,
and using the definition of norm based extensions, we see that $b_s$ defines
a regular section of $B_{\bar X}$  on $V_0 \cap V_1$.
We see immediately that $\varphi(b_s)=s$.
We thus conclude $B_{\bar X}/\Jot_{H^\dagger}$ is locally free $\Oo_{H^\dagger}$-module.

Let us prove (3). 
By choosing a suitable 
products of $s_1$ and $z$, we may create several sections
of $B_{\bar X}$.
Indeed, for each $i \in \Z$, let us define a section $m_{(i)} $ of 
$\Serre{B}{i}$
as follows.
$$
m_{(i)}=s_1^{i}/z^{i }
$$
Then by using an similar argument as above we see that $m_{(i)}$ is a
section of 
$\Serre{B}{i}$
and that it generates 
$\Serre{B}{i}$
as an $\Oo_{H^\dagger}$-module on $V_0 \cap V_1$.

\end{proof}

\subsection{A characterization of Weyl algebras using Auslander regularity}
We now state the second main result of this paper.
\begin{theorem}
Let 
$A=A_n(\kk)$ be  the Weyl algebra over a field $\kk$ 
of positive characteristic $p$.
We assume that the dimension $2 n$ of the base space $X=\Spec(Z(A))$ is greater 
than $2$ (i.e. $n\geq 2$).
Let $B$ be an algebra which is Zariski locally isomorphic to $A$ in the sense
that $B$ is an $Z(A)$-algebra whose associated sheaf $B_X$ on $X$ is Zariski
locally isomorphic to $A_X$.
Then the norm based extension $B_{\bar X}$ of  $B_{X}$
is locally Auslander regular if and only if $B$ is
isomorphic to $A$ as a $Z(A)$-algebra.
\end{theorem}
\begin{proof}
The ``if'' part is a consequence of Proposition \ref{proposition:main}.
Let us prove the ``only if'' part.
We assume that $B_{\bar X}$ is locally Auslander regular.
We first consider the sheaf of algebras $B_{\bar X}/\Iscr_H B_{\bar X}$
which has a filtration given 
by $\{\Serre{B}{k}/\Iscr_H B_{\bar X}\}_{k=0}^{p-1}$.
As we have seen in Lemma \ref{lemma:main},
for each $k\in \{0,1,2,\dots,p-1\}$,
the sheaf  $\gr_k (B_{\bar X}/\Iscr_H B_{\bar X})$ 
is a locally free $\Oo_{H^\dagger}$
module of rank one. 
According to a theorem of Hartshorne(\cite[corollary6.4]{Ha3}),
it is   
isomorphic to a direct sum of invertible $\Oo_H$-modules.
Since the extension groups 
$ \Ext^1_{\Oo_{H}}(\mathcal L_1,\mathcal L_2)$ 
for any invertible sheaves $\mathcal L_1, \mathcal L_2$ on 
$H \isoto \mathbb P^{2 n-1}$ 
which are involved  are all equal to zero, we see that
the whole of $B_{\bar X}/\Iscr_H B_{\bar X}$ is isomorphic to 
a direct sum of invertible
 $\Oo_H$-module.
Then by using the theorem of Horrocks (Theorem \ref{Theorem:abeyoshinaga},
we see that $B_{\bar X}$ is isomorphic
to a direct sum of invertible $\Oo_{\bar X}$-modules.
Namely, we have an isomorphism
$$
B_{\bar X} \isoto 
\bigoplus_{j} \Oo_{\bar X}(c_j) \qquad(\exists c_1,c_2,\dots \in \Z)
$$
as  $\Oo_{\bar X}$-modules. 
The numbers $c_j$ are identical with the ones which appear 
in the splitting of 
$\oplus_{k=0}^{p-1}
\Oo_{H^\dagger}(-k H^\dagger)$ 
over $\Oo_H$ and are in particular independent of $B$.
We thus have 
$$
\dim H^0 (\bar X, B_{\bar X} \otimes \Oo_{\bar X}(k H))
=\dim H^0 (\bar X, A_{\bar X} \otimes \Oo_{\bar X}(k H))
$$
for any integer $k$.
Furthermore, we know by using Theorem \ref{Theorem:splitting} that 
there exists an $\Oo_{\bar X}$-module isomorphism $\Psi$ such that
the following diagram commutes.
$$
\xymatrix{
\Serre{B}{1} 
 \ar[d]^{\Psi}  
\ar[r]
& 
\Serre{B}{1}/B_{\bar X}
\ar[r]^{\phantom{space}\isoto}
& \Oo_{H^\dagger} (1)
\\
\Serre{A}{1}
\ar[r]& 
\Serre{A}{1}
/A_{\bar X}
\ar[ur]^{\isoto} 
}
$$
By a direct calculation (or by using the direct sum decomposition
of $\Serre{A}{1}$), we see that $\Gamma(\bar X, \Serre{A}{1})$ is 
a vector space with a basis  $1,\gamma_1,\gamma_2,\gamma_3,\dots,\gamma_{2n}$.
Accordingly,  
$\Gamma(\bar X, \Serre{B}{1})$ is 
a vector space with a basis  
$1,\Psi(\gamma_1),\Psi(\gamma_2),\Psi(\gamma_3),\dots,\Psi(\gamma_{2n})$.

For each $i,j$, commutator $[\beta_i,\beta_j]$ is apparently a section of 
$H^0(\bar X, \Serre{B}{2})$.

By using the ``asymptotic commutativity'', we see that it is actually a
section of
$H^0(\bar X, B_{\bar X})$.
Since $H^0(\bar X,A_{\bar X})=\kk$, we have $H^0(\bar X,B_{\bar X})=\kk
$ so we conclude that 
there exist some constants $c_{i j} \in \kk$ such that
$$
[\beta_i,\beta_j]=c_{i j}
$$
holds.

It is also easy to see that $\{\beta_i\}$ generates $B$.
Then it follows that $\{\beta_i^p\}$ are in the center of $B$.
$Z(B)=\kk[\beta_1^p,\dots,\beta_{2n}^p].$
It also follows that $\{c_{i j}\}$ is non-degenerate.
$\beta_i^p$ are sections of $\Oo_{\bar X}(H)$. 
By looking at the behavior at $H$, we see
$$
\beta_i^p=\gamma_i^p.
$$
By a suitable linear change of coordinates we obtain a
section $\tilde \beta_i$ of $B$ such that 
$$
[\tilde\beta_i,\tilde\beta_j]=h_{i j}
$$
holds.
The shadow map
should be symplectic. 
So we have $c_{i j}=h_{i j}$.
In other words, $B$ is isomorphic to $A$.

\end{proof}

\bibliography{tsuchimoto}
\bibliographystyle{plain}
\begin{align*}
\phantom{OOOOOOOOOOOOOOOOOOOO}
&\text{Department of mathematics \endgraf} \\
&\text{Kochi University \endgraf} \\
&\text{Akebonocho, Kochi city 780-8520, Japan  \endgraf} \\
&\text{e-mail: docky@kochi-u.ac.jp}
\end{align*}

\end{document}